\documentclass[oneside,english]{amsart}
\usepackage[T1]{fontenc}
\usepackage[latin9]{inputenc}
\usepackage{units}
\usepackage{mathrsfs}
\usepackage{mathtools}
\usepackage{amsbsy}
\usepackage{amstext}
\usepackage{amsthm}
\usepackage{amssymb}
\usepackage{stmaryrd}
\usepackage{stackrel}
\usepackage{graphicx}

\makeatletter

\newcommand{\noun}[1]{\textsc{#1}}

\numberwithin{equation}{section}
\numberwithin{figure}{section}
\theoremstyle{plain}
\newtheorem{thm}{\protect\theoremname}
\theoremstyle{remark}
\newtheorem{rem}[thm]{\protect\remarkname}
\theoremstyle{plain}
\newtheorem{prop}[thm]{\protect\propositionname}
\theoremstyle{plain}
\newtheorem{cor}[thm]{\protect\corollaryname}

\usepackage{babel}

\makeatother

\usepackage{babel}
\providecommand{\corollaryname}{Corollary}
\providecommand{\propositionname}{Proposition}
\providecommand{\remarkname}{Remark}
\providecommand{\theoremname}{Theorem}

\begin{document}
\title{\noun{a geometrical interpretation of okubo spin group}}
\author{Daniele Corradetti $^{\dagger}$, Francesco Zucconi $^{\ddagger}$}
\begin{abstract}
In this work we define, for the first time, the affine and projective
plane over the real Okubo algebra, showing a concrete geometrical
interpretation of its Spin group. Okubo algebra is a flexible, composition
algebra which is also a not unital division algebra. Even though,
Okubo algebra has been known for more than 40 years, we believe that
this is the first time the algebra was used for affine and projective
geometry. After showing that all axioms of affine geometry are verified,
we define a projective plane over Okubo algebra as completion of the
affine plane and directly through the use of Veronese coordinates.
We then present a bijection between the two constructions. Finally
we show a geometric interpretation of $\text{Spin}\left(\mathcal{O}\right)$
as the group of collineations that preserve the axis of the plane. 

\medskip{}

\textbf{Msc}: 17A20; 17A35; 17A75; 51A35. 
\end{abstract}

\maketitle

\section{Introduction and motivations}

Unital composition algebras, also known as Hurwitz algebras, are indispensable
tools in physics. Even more so the four division Hurwitz algebras
that, up to isomorphisms are the Reals $\mathbb{R}$, Complex $\mathbb{C}$,
Quaternions $\mathbb{H}$ and Octonions $\mathbb{O}$. All of them
share the same propriety of \emph{alternativity}, i.e. 
\begin{equation}
\left(x*x\right)*y=x*\left(x*y\right),
\end{equation}
that, combined with the \emph{composition} of the norm, i.e. 
\begin{equation}
n\left(x*y\right)=n\left(x\right)n\left(y\right),
\end{equation}
allows interesting interplays between physics, geometry and algebra.
A wonderful example of those interplays is represented by the Jordan
algebras over three by three matrices with coefficient in the Hurwitz
algebras, whose rank one idempotents are in bijective correspondence
with points on the projective planes\cite{Jordan} and whose automorphisms
are related to collineations of such plane giving rise to numerous
realization of real forms of exceptional Lie groups and algebras \cite{RealF}.

In this work we will start to analyze the geometrical proprieties
of Okubo algebras that are still composition algebras, even though
not unital, and enjoy\emph{ flexibility }instead of alternativity,
i.e. 
\begin{equation}
\left(x*y\right)*x=x*\left(y*x\right).
\end{equation}
The real version of Okubo algebra is even a division algebra and it
was studied by Susumo Okubo while he was working on quark and Gellmann
matrices in the search of an algebra that had $\text{SU}\left(3\right)$
as automorphisms instead of $\text{G}_{2}$ as in the case of Octonions\cite{Okubo95}.
Even more interestingly, Okubo discovered that a deformation of the
product of this algebra gave back the usual Octonionic product\cite{Okubo 1978,Okubo 78c,Okubo1978b}.

More recent works \cite{KMRT} with the joint efforts of Elduque and
Myung \cite{Elduque 91,Elduque 93,Elduque Myung 90,ElduQueAut} helped
to clarify the context of Okubo algebras that are part of a larger
set of algebras called \emph{symmetric composition algebra }that hosts
togheter with Okubo algebras and its split version all the para-Hurwitz
algebras, i.e. all the algebras obtained from Hurwitz algebras defining
a new product 
\begin{equation}
x*y=\overline{x}\overline{y}.
\end{equation}
Later on some algebrical features of Okubo algebra were investigated
in order to construct all real form of exceptional Lie groups obtaining
a new version of Freudenthal Magic Square\cite{EldMS1,EldMS2}. Nevertheless,
a purely geometric construction over Okubo algebra was, to our knowledge,
never fully investigated nor even defined. On the other side, in \cite{ElDuque Comp}
Elduque gave a description and a definition of the Spin group over
Okubo algebras $\text{Spin}\left(\mathcal{O}\right)$ and their triality
algebras $\mathfrak{tri}\left(\mathcal{O}\right)$. Here we will give
a geometric interpretation of $\text{Spin}\left(\mathcal{O}\right)$,
starting from the definition of the Okubic affine plane.

In this work we define for the first time the affine and the projective
plane over the real Okubo algebra and investigate some of their geometrical
features and interplays. In section $2$ we introduce the real Okubo
algebra, while in section $3$ and $4$we introduce the affine and
projective plane defining points, lines and elliptic and hyperbolic
polarities. In section $5$ we analyze collineations of the plane
with a special attention to the triality collineation and, finally,
we give a geometric interpretation of $\text{Spin}\left(\mathcal{O}\right)$,
which was already defined algebrically in \cite{ElduQueAut}, as the
group of collineations that preserve the axis of the plane.

\section{Okubo Algebras and Octonions}

Following \cite{Okubo 1978} and \cite{Elduque Myung 90}, we define
the Okubo Algebra $\mathcal{O}$ as the set of three by three Hermitian
traceless matrices over the complex numbers $\mathbb{C}$ with the
following product 
\begin{equation}
x*y=\mu\cdot xy+\overline{\mu}\cdot yx-\frac{1}{3}\text{Tr}\left(xy\right),
\end{equation}
where $\mu=\nicefrac{1}{6}\left(3+\text{i}\sqrt{3}\right)$ and the
juxtaposition is the ordinary associative product between matrices.
It is easy to see that the resulting algebra is not unital, not associative
and not alternative. Nevertheless, $\mathcal{O}$ is a\emph{ flexible
}algebra, i.e. 
\begin{equation}
x*\left(y*x\right)=\left(x*y\right)*x,
\end{equation}
which will turn out to be more useful property than alternativity
in the definition of the projective plane. Even though the Okubo algebra
is not unital, it does have idempotents, i.e. $e*e=e$, such as 
\begin{equation}
e=\left(\begin{array}{ccc}
2 & 0 & 0\\
0 & -1 & 0\\
0 & 0 & -1
\end{array}\right),
\end{equation}
that together with

\begin{align}
\text{i}_{1}=\left(\begin{array}{ccc}
0 & 1 & 0\\
-1 & 0 & 0\\
0 & 0 & 0
\end{array}\right) & ,\,\,\,\,\text{i}_{2}=\left(\begin{array}{ccc}
0 & 0 & 1\\
0 & 0 & 0\\
-1 & 0 & 0
\end{array}\right),\nonumber \\
\text{i}_{3}=\left(\begin{array}{ccc}
0 & 0 & 0\\
0 & 0 & 1\\
0 & -1 & 0
\end{array}\right) & ,\,\,\,\,\text{i}_{4}=\left(\begin{array}{ccc}
0 & i & 0\\
i & 0 & 0\\
0 & 0 & 0
\end{array}\right),\\
\text{i}_{5}=\left(\begin{array}{ccc}
0 & 0 & i\\
0 & 0 & 0\\
i & 0 & 0
\end{array}\right), & \,\,\,\,\,\,\text{i}_{6}=\left(\begin{array}{ccc}
0 & 0 & 0\\
0 & 0 & i\\
0 & i & 0
\end{array}\right),\text{i}_{7}=\left(\begin{array}{ccc}
0 & 0 & 0\\
0 & i & 0\\
0 & 0 & -i
\end{array}\right),\label{eq:definizione i ottonioniche}
\end{align}
form a basis for $\mathcal{O}$ that has real dimension $8$. 
\begin{rem}
The Okubo algebra we defined on Hermitian traceless matrices and with
real dimension $8$ is called \emph{real}, while the one defined over
$\mathfrak{sl}_{3}\left(\mathbb{C}\right)$ has complex dimension
$8$, is not a division algebra and is called \emph{complex} Okubo
algebra. In this work we will only account for the real Okubo algebra,
which, as shown in Prop.\ref{prop:division} is a division algebra.
In a forthcoming paper we study geometryc structure over a complex
Okubo algebra. 
\end{rem}

Let us consider the quadratic form $n$ over Okubo algebra, given
by 
\begin{equation}
n\left(x\right)=\frac{1}{6}\text{Tr}\left(x^{2}\right).
\end{equation}
It is straightforward to see that this is a \emph{norm} over $\mathcal{O}$
with respect to which the real Okubo algebra turns to be a composition
algebra, i.e. 
\begin{align}
n\left(x*y\right) & =n\left(x\right)n\left(y\right),\\
\left\langle x*y,z\right\rangle  & =\left\langle x,y*z\right\rangle ,\label{eq:symmetric polar}
\end{align}
where $\left\langle \cdot,\cdot\right\rangle $ is the \emph{polar
form} given by 
\begin{equation}
\left\langle x,y\right\rangle =n\left(x+y\right)-n\left(x\right)-n\left(y\right).\label{eq:polar form}
\end{equation}

Algebras that are flexible and composition such as Okubo are called
from \cite[Ch. VIII]{KMRT} \emph{symmetric composition} algebras
and enjoy the following notable relation 
\begin{equation}
x*\left(y*x\right)=\left(x*y\right)*x=n\left(x\right)y.
\end{equation}

For our purposes it is of paramount importance to notice the following
\cite{OkMy80} 
\begin{prop}
\label{prop:division}The Okubo Algebra is a division algebra 
\end{prop}

\begin{proof}
Without any loss of generality, let us suppose that $d\neq0$ is a
left divisor of zero, i.e. $d*x=0$, then 
\[
n\left(d*x\right)=n\left(d\right)n\left(x\right)=0.
\]
But, since the algebra is a division algebra and flexible, we have
also have

\begin{align}
\left(d*x\right)*d & =0=n\left(d\right)x,
\end{align}
and therefore $n\left(d\right)=0$, i.e. $\text{Tr}\left(d^{2}\right)=0$.
But since the element $d\in\mathcal{O}$ is of the form 
\[
d=\left(\begin{array}{ccc}
\xi_{1} & d_{3} & \overline{d}_{2}\\
\overline{d}_{3} & \xi_{2} & d_{1}\\
d_{2} & \overline{d}_{1} & -\xi_{1}-\xi_{2}
\end{array}\right),
\]
$\text{Tr}\left(d^{2}\right)=0$ is not possible with $\xi_{1},$$\xi_{2}\in\mathbb{R}$
and $\xi_{1},\xi_{2}\neq0$. 
\end{proof}
Unfortunately, since $\mathcal{O}$ is not a unital algebra, an element
$x$ does not have defined an inverse. Nevertheless, considering the
existance of the idempotent $e$, and inspired by the identity 
\[
x*\left(e*x\right)=\left(x*e\right)*x=n\left(x\right)e,
\]
we can define $\left(x\right)_{L}^{-1}=n\left(x\right)^{-1}\left(e*x\right)$
and $\left(x\right)_{R}^{-1}=n\left(x\right)^{-1}\left(x*e\right)$
so that 
\[
\left(x\right)_{L}^{-1}*x=x*\left(x\right)_{R}^{-1}=e.
\]

As an implication of this we have the following 
\begin{prop}
Any equation of the kind 
\begin{equation}
a*x=b,\,\,\text{or}\,\,\,\,x*a=b,
\end{equation}
has a unique solution which is respectively given by 
\begin{equation}
x=\frac{1}{n\left(a\right)}b*a,\,\,\,\text{or}\,\,\,\,x=\frac{1}{n\left(a\right)}a*b.
\end{equation}
\end{prop}

Even though, extremely simple the previous proposition has paramount
geometrical implications as we show in section 3.
\begin{rem}
Since $\mathcal{O}$ is not unital, there is not a canonical involution,
as in Hurwitz algebras, of the type $\overline{x}=\left\langle x,\boldsymbol{1}\right\rangle \boldsymbol{1}-x$
such that $n\left(x\right)=x\overline{x}=\overline{x}x$. Nevertheless,
the idempotent $e$ gives rise to an automorphism of order two, i.e.
an involution, $x\longrightarrow\left\langle x,e\right\rangle e-x$
and an automorphism of order three $x\longrightarrow\tau\left(x\right)=\left\langle x,e\right\rangle e-x*e$,
whose geometrical meaning was partially investigated in \cite{ElDuque Comp}. 
\end{rem}

\subsection{Deformation to Octonions}

The choice of the idempotent $e$ also allows a deformation on the
Okubic product that give rise to the usual octonionic product. Indeed,
defining a new product over $\mathcal{O}$ by 
\begin{equation}
x\cdot y=\left(e*x\right)*\left(y*e\right),
\end{equation}
we obtain that the resulting algebra $\left(\mathcal{O},\cdot,n\right)$
is isomorphic to that of Octonions $\mathbb{O}$, with $e$ as the
unit element. Following \cite{ElduQueAut}, it is easy to show that
since $e*e=e$ and $n\left(e\right)=1$, for every $x\in\mathcal{O}$
the element $e$ acts as a left and right identity, i.e. 
\begin{align}
x\cdot e & =e*x*e=n\left(e\right)x=x,\\
e\cdot x & =e*x*e=n\left(e\right)x=x.
\end{align}
Moreover, since Okubo algebra is a composition algebra, the same norm
$n$ enjoys the following relation 
\begin{equation}
n\left(x\cdot y\right)=n\left(\left(e*x\right)*\left(y*e\right)\right)=n\left(x\right)n\left(y\right),
\end{equation}
which means that $\left(\mathcal{O},\cdot,n\right)$ is a unital composition
algebra of real dimension $8$ and therefore, as noted by Okubo himself
\cite{Okubo 1978,Okubo 78c,ElduQueAut}, is isomorphic to the algebra
of Octonions $\mathbb{O}$. 
\begin{rem}
Under the previous hypothesis the octonionic identity is $1=e$, while
the imaginary units on the Octonions are $\left\{ \text{i}_{1},...,\text{i}_{7}\right\} $
in (\ref{eq:definizione i ottonioniche}). Okubo algebras and Octonionic
algebras are strictly intertwined: the choice of an idempotent on
the Okubo algebra allows the definition of an octonionic product and
of a conjugation that can be explicitly given by 
\begin{align}
x\cdot y & =\left(e*x\right)*\left(y*e\right),\\
\overline{x} & =e*\left(e*\left(e*x\right)\right).
\end{align}
\end{rem}

\subsection{Automorphisms}

The first interest of Okubo in its algebra \cite{Okubo 1978,ElduQueAut}
was given by the fact that the group of automorphisms of $\mathcal{O}$
is $\text{Aut}\left(\mathcal{O}\right)\cong SU\left(3\right)$
and its derivations is $\mathfrak{der}\left(\mathcal{O}\right)\cong\mathfrak{su}\left(3\right)$,
while in the octonionic case the automorphisms and the derivation
led to exceptional group $G_{2\left(-14\right)}$ and $\mathfrak{g}_{2}$.
But since $\mathcal{O}$ is a symmetric composition algebra automorphisms
of $\mathcal{O}$ it also enjoy the property of preserving the norm.
Indeed, let $\varphi$ be an automorphism of $\mathcal{O}$. Then
for $x,y\in\mathcal{O}$, since $\varphi\left(x*y\right)=\varphi\left(x\right)*\varphi\left(y\right)$,
then we have 
\begin{align}
\varphi\left(\left(x*y\right)*x\right) & =\varphi\left(n\left(x\right)y\right),
\end{align}
but on the other side we also have 
\begin{align}
\varphi\left(\left(x*y\right)*x\right) & =\left(\varphi\left(x\right)*\varphi\left(y\right)\right)*\varphi\left(x\right)=\\
 & =n\left(\varphi\left(x\right)\right)\varphi\left(y\right),
\end{align}
and therefore any automorphism of $\mathcal{O}$ preserves the norm,
i.e. 
\begin{equation}
n\left(x\right)=n\left(\varphi\left(x\right)\right),
\end{equation}
and thus belongs to the orthogonal group $O\left(\mathcal{O}\right)$
and $\left\langle x,y\right\rangle =\left\langle \varphi\left(x\right),\varphi\left(y\right)\right\rangle $
which also means that $\text{tr}\left(xy\right)=\text{tr}\left(\varphi\left(xy\right)\right)$.

\section{the affine plane over real okubo algebras}
\begin{center}
\begin{figure}
\centering{}\includegraphics[scale=0.3]{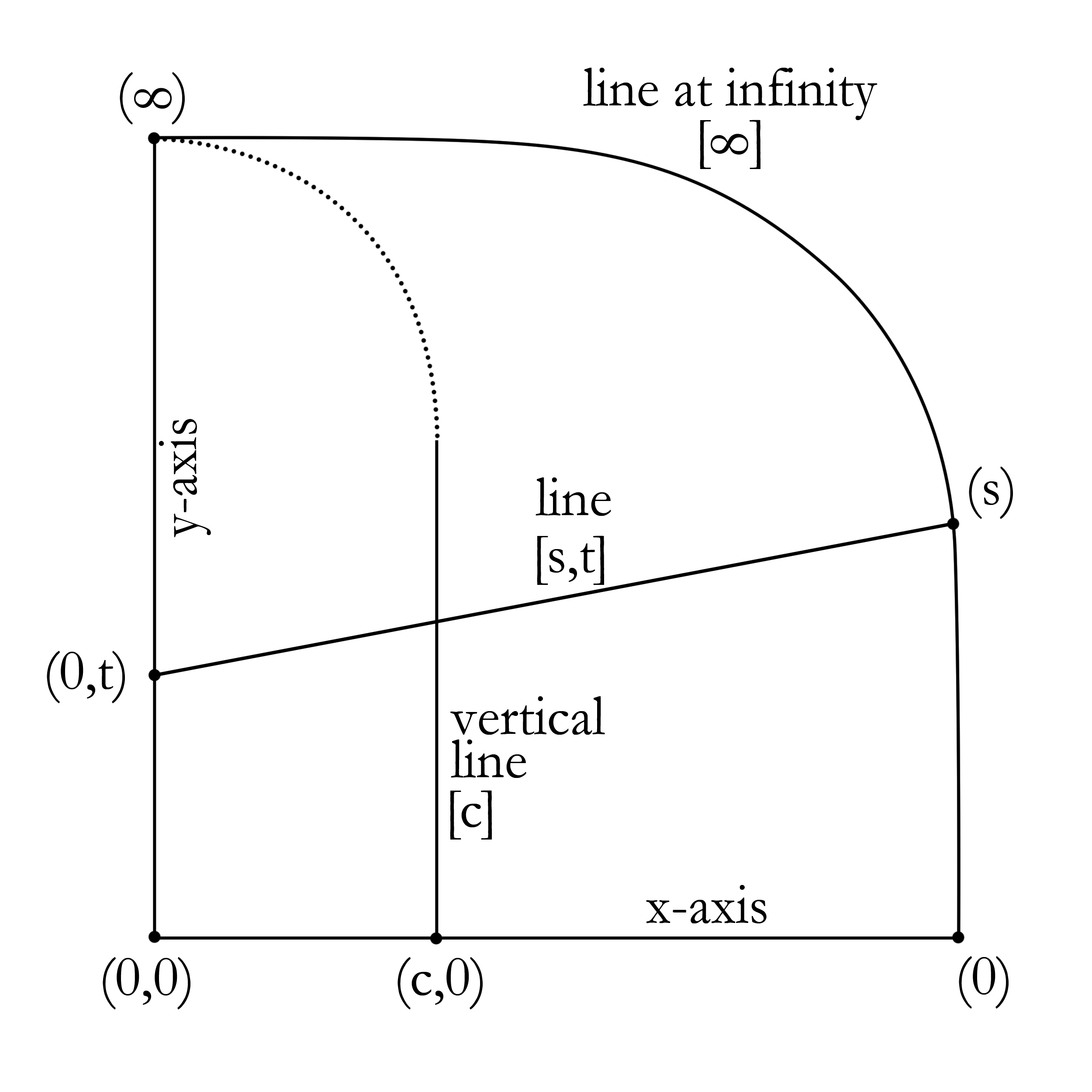}\caption{\label{fig:The affine plane}Representation of the affine plane: $\left(0,0\right)$
represents the origin, $\left(0\right)$ the point at the infinity
on the $x$-axis, $\left(s\right)$ is the point at infinity of the
line $\left[s,t\right]$ of slope $s$ while $\left(\infty\right)$
is the point at the infinity on the $y$-axis and of vertical lines
$\left[c\right]$.}
\end{figure}
\par\end{center}

Following \cite{Compact Projective} we now define Okubic affine plane
$\mathscr{A}_{2}\left(\mathcal{O}\right)$ that we will later complete
and identify with the projective plane $\mathbb{P}^{2}\mathcal{O}$.
A \emph{point} on the affine plane is identified by two coordinates
$\left(x,y\right)$ with $x,y\in\mathcal{O}$, while a \emph{line}
of the affine plane the set $\left\{ \left(x,s*x+t\right):x\in\mathcal{O}\right\} ,$
that we label $\left[s,t\right]$ where $s,t\in\mathcal{O}$ are the
slope and the offset respectively. Vertical lines are identified by
$\left[c\right]$ and denote the set $\left\{ c\right\} \times\mathcal{O}$,
where $c\in\mathcal{O}$ represent the intersection with the $x$
axis.

Since $\mathcal{O}$ is a division algebra we then have the following
properties: 
\begin{enumerate}
\item for any two points $\left(x_{1},y_{1}\right)$ and $\left(x_{2},y_{2}\right)$
there is a unique line joining, namely $\left[s,y_{1}-s*x_{1}\right]$,
where $s$ is determined by 
\[
s*\left(x_{1}-x_{2}\right)=\left(y_{1}-y_{2}\right),
\]
and explicitly found through the use of flexibility 
\[
\left(x_{1}-x_{2}\right)*s*\left(x_{1}-x_{2}\right)=\left(x_{1}-x_{2}\right)*\left(y_{1}-y_{2}\right),
\]
that yields to 
\[
s=\frac{\left(x_{1}-x_{2}\right)*\left(y_{1}-y_{2}\right)}{n\left(x_{1}-x_{2}\right)},
\]
when $x_{1}\neq x_{2}$ and to the line $\left[x_{1}\right]$ when
$x_{1}=x_{2}$. 
\item Two lines $\left[s_{1},t_{1}\right]$ and $\left[s_{2},t_{2}\right]$
of different slope, i.e. $s_{1}\neq s_{2}$, have a unique point of
intersection $\left\{ \left(x,s_{1}*x+t_{1}\right)\right\} $ where
\[
x=\frac{\left(s_{1}-s_{2}\right)*\left(t_{1}-t_{2}\right)}{n\left(s_{1}-s_{2}\right)}.
\]
\item Two lines with the same slope are disjoint and therefore are called
\emph{parallel}. 
\item For each line $\left[s,t\right]$ and each point $\left(x,y\right)$
there is a unique line which passes through $\left(x,y\right)$ and
is parallel to $\left[s,t\right]$, i.e. $\left[s,y-s*x\right]$. 
\end{enumerate}
To achieve a completion $\overline{\mathscr{A}_{2}}\left(\mathcal{O}\right)$
of the affine plane, we add another set of coordinates which will
represent the points at infinity. Indeed, we add a line at infinity
$\left[\infty\right]$, i.e. 
\[
\left[\infty\right]=\left\{ \left(s\right):s\in\mathcal{O}\cup\left\{ \infty\right\} \right\} ,
\]
where $\left(s\right)$ identify the end point at infinity of a line
with slope $s\in\mathcal{O}\cup\left\{ \infty\right\} $. Finally,
we define $\left(\infty\right)$ the point at infinity of $\left[\infty\right]$.
It is easy to verify that this construction preserves the property
of a unique line joining two different points and that two lines intersect
at infinity if and only if they are parallel, i.e. have the same slope
$\left(s\right)$.

Resuming the whole notation, as in Fig.\ref{fig:The affine plane},
we have three set of coordinates that indentify all the points in
the completion of the affine plane, i.e. $\left(x,y\right),\left(s\right)$
and $\left(\infty\right)$. The whole affine plane is encompassed
by a triangle given by three special points: the \emph{origin} $\left(0,0\right)$;
the \emph{0-point at infinity}, i.e. the point $\left(0\right)$ obtained
prolonging the $x$ axis to infinity; finally, the \emph{$\infty$-point
at infinity}, i.e. the point $\left(\infty\right)$ obtained prolonging
the $y$ axis to infinity. We will call $\triangle$ the set made
by those three points, i.e. $\triangle=\left\{ \left(0,0\right),\left(0\right),\left(\infty\right)\right\} .$

\section{Projective plane over the okubo algebra}

We here define the projective plane $\mathbb{P}^{2}\mathcal{O}$ making
use of an \emph{ad hoc} adaptation of the Veronese coordinates introduced
in \cite{Compact Projective,Notes Octo} for the definition of the
Octonionic plane and used in \cite{RealF} for the Bioctonionic Rosenfeld
plane. After a suitable definition of points and lines of the plane
it is straightforward the notion of polarity, which guarantees the
projective ``spirit'' of the construction. Finally, we show that
the resulting projective plane is in fact another way of seeing the
completion of the affine plane $\overline{\mathscr{A}_{2}}\left(\mathcal{O}\right)$.

\subsection{The projective plane }

Let $V\cong\mathbb{\mathcal{O}}^{3}\times\mathbb{R}^{3}$ be a real
vector space, with elements of the form 
\[
\left(x_{\nu};\lambda_{\nu}\right)_{\nu}=\left(x_{1},x_{2},x_{3};\lambda_{1},\lambda_{2},\lambda_{3}\right)
\]
where $x_{\nu}\in\mathcal{O}$, $\lambda_{\nu}\in\mathbb{R}$ and
$\nu=1,2,3$. A vector $w\in V$ is called \emph{Okubic Veronese}
if

\begin{align}
\lambda_{1}x_{1} & =x_{2}*x_{3},\,\,\lambda_{2}x_{2}=x_{3}*x_{1},\,\,\lambda_{3}x_{3}=x_{1}*x_{2}\label{eq:Ver-1-1}\\
n\left(x_{1}\right) & =\lambda_{2}\lambda_{3},\,n\left(x_{2}\right)=\lambda_{3}\lambda_{1},n\left(x_{3}\right)=\lambda_{1}\lambda_{2}.\label{eq:Ver-2-1}
\end{align}
Now we will consider the subspace $H\subset V$ be of Okubo-Veronese
vectors. It is straightforward to see that if $w=\left(x_{\nu};\lambda_{\nu}\right)_{\nu}$
is an Okubic Veronese vector then also $\mu w=\mu\left(x_{\nu};\lambda_{\nu}\right)_{\nu}$
is such vector and therefore $\mathbb{R}w\subset H$. We define the
\emph{Okubic projective plane} $\mathbb{P}^{2}\mathcal{O}$ as the
geometry having this 1-dimensional subspaces $\mathbb{R}w$ as points,
i.e. 
\begin{equation}
\mathbb{P}^{2}\mathcal{O}=\left\{ \mathbb{R}w:w\in H\smallsetminus\left\{ 0\right\} \right\} .
\end{equation}

\subsection{Lines}

The norm on $\mathbb{\mathcal{O}}$ defines the polar form (\ref{eq:polar form})
that can be extended on $V$ as a symmetric bilinear form 
\begin{equation}
\beta\left(v,w\right)=\stackrel[\nu=1]{3}{\sum}\left(\left\langle x_{\nu},y_{\nu}\right\rangle +\lambda_{\nu}\eta_{\nu}\right),
\end{equation}
where $v=\left(x_{\nu};\lambda_{\nu}\right)_{\nu}$ and $w=\left(y_{\nu};\eta_{\nu}\right)_{\nu}$
are Okubo-Veronese vectors in $H\subset V\cong\mathbb{\mathcal{O}}^{3}\times\mathbb{R}^{3}$.
Therefore, we define the \emph{lines} $\ell_{w}$ in the projective
plane $\mathbb{P}^{2}\mathcal{O}$ as the orthogonal spaces of a vector
$w\in H$, i.e. 
\begin{equation}
\ell_{w}=w^{\bot}=\left\{ z\in V:\beta\left(z,w\right)=0\right\} ,
\end{equation}
and, clearly, a point $\mathbb{R}v$ is incident to the line $\ell_{w}$
when $\mathbb{R}v\subseteq$$w^{\bot}$. 
\begin{rem}
It is worth expliciting how the norm $n$ defined over the symmetric
composition algebra $\mathbb{\mathcal{O}}$ is intertwined with the
geometry of the plane. This relations is manifest when we consider
the quadratic form of the bilinear symmetric form $\beta$, i.e.

\begin{equation}
q\left(v\right)\coloneqq\frac{1}{2}\beta\left(v,v\right)=n\left(x_{1}\right)+n\left(x_{2}\right)+n\left(x_{3}\right)+\frac{1}{2}\left(\lambda_{1}^{2}+\lambda_{2}^{2}+\lambda_{3}^{2}\right),
\end{equation}
\end{rem}

\subsection{Polarity}

Given the previous definitions, it is straightforward to define the
\emph{standard elliptic polarity} $\pi$ that maps points into lines
and lines into points through orthogonality, i.e. 
\begin{equation}
\pi\left(w\right)=w^{\perp},\pi\left(w^{\perp}\right)=w,
\end{equation}
where $w$ is an Okubo-Veronese vector in $V$.

\subsection{Correspondence with the affine plane}

The identification of the affine Okubic plane with the projective
can be explicited defining the map that sends a point of the affine
plane to the projective point in $V\cong\mathbb{\mathcal{O}}^{3}\times\mathbb{R}^{3}$,
i.e. 
\begin{align}
\left(x,y\right) & \mapsto\mathbb{R}\left(x,y,x*y;n\left(y\right),n\left(x\right),1\right),\label{eq:corrispondenza1}\\
\left(x\right) & \mapsto\mathbb{R}\left(0,0,x;n\left(x\right),1,0\right),\\
\left(\infty\right) & \mapsto\mathbb{R}\left(0,0,0;1,0,0\right).
\end{align}
Since the Okubo algebra is a symmetric composition algebra, then the
map is well defined. Indeed, from (\ref{eq:Ver-2-1}) we note that
\begin{equation}
n\left(x\right)=\lambda_{2},\,\,\,n\left(y\right)=\lambda_{1},
\end{equation}
and since Okubo is a composition algebra then 
\begin{equation}
n\left(x*y\right)=n\left(x\right)n\left(y\right).
\end{equation}
Since Okubo algebra is flexible we also have that (\ref{eq:Ver-1-1})
are satisfied and 
\begin{align*}
\lambda_{1}x & =y*\left(x*y\right)=n\left(y\right)x,\,\,\\
\lambda_{2}y & =\left(x*y\right)*x=n\left(x\right)y,\\
\lambda_{3}\left(x*y\right) & =x*y,
\end{align*}
and therefore $\mathbb{R}\left(x,y,x*y;n\left(y\right),n\left(x\right),1\right)$
is a point in the Okubic projective plane. As for the converse, if
a point $p$ of coordinates $\left(x_{\nu};\lambda_{\nu}\right)_{\nu}$
is in $\mathbb{P}^{2}\mathcal{O}$ then satisfy (\ref{eq:Ver-2-1})
and has one of the $\lambda_{\nu}$ different from zero. Let us suppose
$\lambda_{3}=1$. Then by (\ref{eq:Ver-2-1}) we have that $x_{3}=x_{1}*x_{2}$
and therefore the point $p$ is of form $\mathbb{R}\left(x_{1},x_{2},x_{1}*x_{2};n\left(x_{2}\right),n\left(x_{1}\right),1\right)$
and, therefore, corresponds to $\left(x_{1},x_{2}\right)$ in the
affine plane. If $\lambda_{3}=0$ and $\lambda_{2}=1$ then the point
is of the form $\mathbb{R}\left(0,0,x;n\left(x\right),1,0\right)$
and therefore corresponds to the affine point $\left(x\right)$, while
if $\lambda_{1}=1$ and $\lambda_{3}=\lambda_{2}=0$ then the point
is $\mathbb{R}\left(0,0,0;1,0,0\right)$ and therefore corresponds
to $\left(\infty\right)$.

The same reasoning shows that the correspondence between affine and
projective lines given by

\begin{align}
\left[s,t\right] & \mapsto\left(t*s,-t,-s;1,n\left(s\right),n\left(t\right)\right)^{\bot},\\
\left[c\right] & \mapsto\left(-c,0,0;0,1,n\left(c\right)\right)^{\bot},\\
\left[\infty\right] & \mapsto\left(0,0,0;0,0,1\right)^{\bot},
\end{align}
is also a bijection.

Finally, we need to verify that the image of a point $\left(x,y\right)$
incident to the line $\left[s,t\right]$ goes into a point of the
projective plane, i.e. $\mathbb{R}\left(x,y,x*y;n\left(y\right),n\left(x\right),1\right)$,
that is incident to the image of the same projective line, i.e. $\left(t*s,-t,-s;1,n\left(s\right),n\left(t\right)\right)^{\bot}$.
By definition, for the image of $\left(x,y\right)$ to be incident
to the image of $\left[s,t\right]$, the following condition must
be satisfied 
\begin{equation}
\left\langle t*s,x\right\rangle -\left\langle t,y\right\rangle -\left\langle s,x*y\right\rangle +n\left(y\right)+n\left(s\right)n\left(x\right)+n\left(t\right)=0.\label{eq:eqretta}
\end{equation}
Noting that 
\begin{equation}
\left\langle s*x,t-y\right\rangle =n\left(s*x+t-y\right)-n\left(s*x\right)-n\left(t-y\right),
\end{equation}
and since (\ref{eq:symmetric polar}), we then have
\begin{align}
\left\langle s*x,t-y\right\rangle  & =\left\langle s*x,t\right\rangle -\left\langle s*x,y\right\rangle \\
 & =\left\langle t,s*x\right\rangle -\left\langle s,x*y\right\rangle \\
 & =\left\langle t*s,x\right\rangle -\left\langle s,x*y\right\rangle ,
\end{align}
and, tehrefore, 
\begin{equation}
\left\langle t*s,x\right\rangle -\left\langle s,x*y\right\rangle =n\left(s*x+t-y\right)-n\left(s*x\right)-n\left(t-y\right).
\end{equation}
 Inserting the latter into (\ref{eq:eqretta}) and noting that $n\left(s\right)n\left(x\right)=n\left(s*x\right)$
we then have that (\ref{eq:eqretta}) is equivalent to 
\begin{align}
n\left(s*x+t-y\right) & =0,
\end{align}
and, since we are in a division composition algebra, to the condition
$s*x+t-y=0$. This means that the above bijection sends affine points
incident to an affine line in projective point incident to the same
projective line. 

\section{Collineations and the Spin group}

\subsection{Collineations}

In analogy to the octonionic case, we are now interested in studying
the collineations on the completion of Okubic affine plane that we
will define as transformations of the plane that send a line into
another line, i.e. $\varphi\left(\left[s,t\right]\right)=\left[s',t'\right]$
with $s,t,s',t'\in\mathbb{\mathcal{O}}$.

Obviously the set of collineations forms a group under composition
and since the identity is a collineation itself, the group is not
void. Moreover, throught the use of the Okubic-Veronese coordinates
an order three element of the group can be easily spotted, i.e. the
\emph{triality collineation} \cite{Compact Projective} given by a
cyclic permutation of the coordinates 
\begin{equation}
\widetilde{t}:\left(x_{1},x_{2},x_{3};\lambda_{1},\lambda_{2},\lambda_{3}\right)\longrightarrow\left(x_{2},x_{3},x_{1};\lambda_{2},\lambda_{3},\lambda_{1}\right).
\end{equation}

\begin{prop}
The triality collineation can be red on the affine plane in the following
way:

\begin{equation}
\widetilde{t}:\begin{cases}
\left(x,y\right) & \longrightarrow\frac{1}{n\left(y\right)}\left(y,x*y\right),\,\,\,y\neq0\\
\left(x\right) & \longrightarrow\frac{1}{n\left(x\right)}\left(0,x\right),x\neq0\\
\left(x,0\right) & \longrightarrow\left(x\right),\,\,\,\\
\left(0\right) & \longrightarrow\left(\infty\right),\\
\left(\infty\right) & \longrightarrow\left(0,0\right).
\end{cases}\label{eq:triality}
\end{equation}
 In particular it induces a collineation $t\colon\mathscr{A}_{2}\left(\mathcal{O}\right)\rightarrow\mathscr{A}_{2}\left(\mathcal{O}\right)$on
the affine plane.
\end{prop}

\begin{proof}
If $y\neq0$, the image of $t\left(x,y\right)$ with through the bijection
(\ref{eq:corrispondenza1}) in the projective plane is given by 
\begin{equation}
\frac{1}{n\left(y\right)}\left(y,x*y\right)\longrightarrow\frac{1}{n\left(y\right)}\left(y,x*y,\frac{y*x*y}{n\left(y\right)};\frac{n\left(x*y\right)}{n\left(y\right)},1,n\left(y\right)\right),
\end{equation}
and since $y*x*y=n\left(y\right)x$ and $n\left(x*y\right)=n\left(x\right)*n\left(y\right)$,
then the image of $t\left(x,y\right)$ is in $\mathbb{R}\left(y,x*y,x;n\left(x\right),1,n\left(y\right)\right)$
which is the image of the triality collineation $\widetilde{t}$ of
the projective point $\mathbb{R}\left(x,y,x*y;n\left(y\right),n\left(x\right),1\right)$.
With the same procedure we find the other correspondences. 

\begin{figure}

\begin{centering}
\includegraphics[scale=0.14]{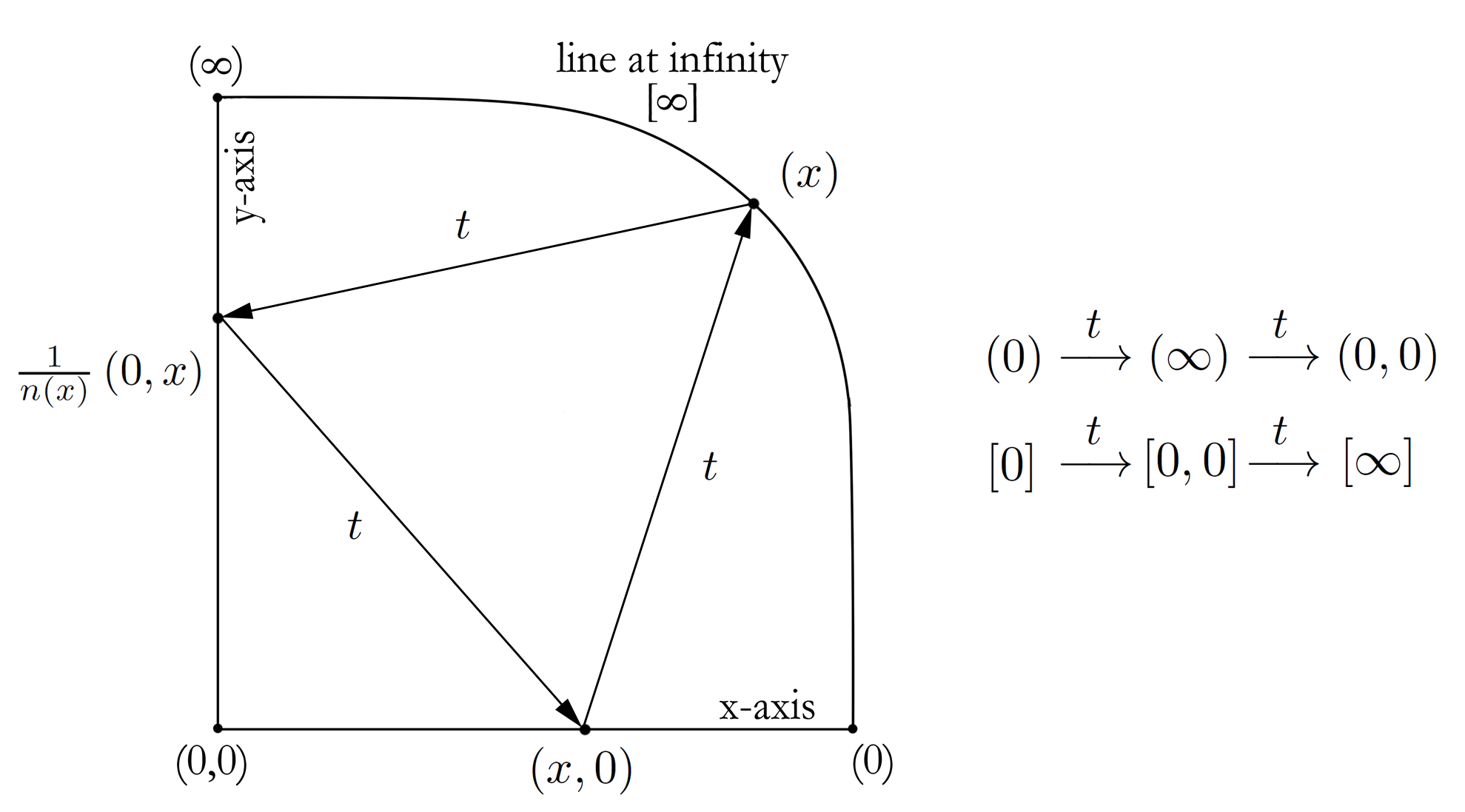}\caption{\label{fig:Action-on-the}Action on the affine plane $\mathscr{A}_{2}\left(\mathcal{O}\right)$
of the triality collineation defined in (\ref{eq:triality}).}
\par\end{centering}
\end{figure}
\end{proof}
\begin{rem}
As shown in Fig (\ref{fig:Action-on-the}) the triality collineation
$t$ sends the line at infinity $\left[\infty\right]$ into the line
$\left[0\right]$, while the $y$ axis $\left[0\right]$ is sent into
the $x$ axis $\left[0,0\right]$; finally the $x$ axis $\left[0,0\right]$
is sent into the line at infinity $\left[\infty\right]$. This phenomenon
is the dual, of what happens, in the reverse order, for the three
points $\left(0,0\right)$,$\left(0\right)$ and $\left(\infty\right)$. 
\end{rem}

\subsection{Recovering the Spin group}

While in the previous section we analysed a subgroup of collineations
that cyclically permutated the $x-$axis, the $y-$axis and the line
at infinity, now we are interested in the subgroup $\Gamma\left(\triangle,\mathbb{\mathcal{O}}\right)$
composed by collineations that fix every point of the triangle $\triangle$,
i.e. $\varphi\left(\left(0,0\right)\right)=\left(0,0\right)$, $\varphi\left(\left(0\right)\right)=\left(0\right)$
and $\varphi\left(\left(\infty\right)\right)=\left(\infty\right)$,
or, in other words, that fix the $x$ and $y$ axis and, therefore,
the line at infinity. This will allow to dive a geometric interpretation
of the algebrical definition given in \cite{ElduQueAut} of the Sping
group over the Okubo algebra, i.e. $\text{Spin}\left(\mathcal{O}\right)$. 
\begin{prop}
The group $\Gamma\left(\triangle,\mathbb{\mathcal{O}}\right)$ of
collineations that fix every point of $\triangle$ are transformations
of this form 
\begin{align}
\left(x,y\right) & \mapsto\left(A\left(x\right),B\left(y\right)\right)\\
\left(s\right) & \mapsto\left(C\left(s\right)\right)\\
\left(\infty\right) & \mapsto\left(\infty\right)
\end{align}
where $A,B$ and $C$ are automorphism in respect to the sum over
$\mathbb{\mathcal{O}}$ and in respect to multiplication they satisfy
\begin{equation}
B\left(s*x\right)=C\left(s\right)*A\left(x\right).
\end{equation}
\end{prop}

\begin{proof}
A collineation that fixes $\left(0,0\right)$, $\left(0\right)$ and
$\left(\infty\right)$, it also fixes the $x$-axis and $y$-axis
and all lines that are parallel to them. This means that the first
coordinate is the image of a function that does not depend on $y$
and the second coordinate is image of a fuction that does not depend
of $x$, i.e. $\left(x,y\right)\mapsto\left(A\left(x\right),B\left(y\right)\right)$
and $\left(s\right)\mapsto\left(C\left(s\right)\right)$. Now consider
the image of a point on the line $\left[s,t\right]$. The point is
of the form $\left(x,s*x+t\right)$ and its image goes to 
\begin{equation}
\left(x,s*x+t\right)\mapsto\left(A\left(x\right),B\left(s*x+t\right)\right).
\end{equation}
In order this to be a collineation, the points of $\left[s,t\right]$
must all belong to a line that, setting $x=0$, passes through the
points $p_{1}=\left(0,B\left(t\right)\right)$ and $p_{2}=\left(C\left(s\right)\right)$,
e.g. $\left[C\left(s\right),B\left(t\right)\right]$. Every line $\left(A\left(x\right),B\left(s*x+t\right)\right)$
passing through $p_{1}$ and $p_{2}$ must satisfy the condition 
\begin{equation}
B\left(s*x+t\right)=C\left(s\right)*A\left(x\right)+B\left(t\right).\label{eq:Condition B=00003D00003DCA}
\end{equation}
Given (\ref{eq:Condition B=00003D00003DCA}), if $B$ is an automorphism
with respect to the sum over $\mathbb{\mathcal{O}}$, then $B\left(s*x\right)=C\left(s\right)*A\left(x\right)$.
Conversely if $B\left(s*x\right)=C\left(s\right)*A\left(x\right)$
is true than $B\left(s*x+t\right)=B\left(s*x\right)+B\left(t\right)$
and $B$ is an automorphism with respect to the sum. 
\end{proof}
Let us consider the quadrangle $\boxempty$ given by the points $\left(0,0\right)$,
$\left(e,e\right)$, $\left(0\right)$ and $\left(\infty\right)$,
that is $\boxempty=\triangle\cup\left\{ \left(e,e\right)\right\} $,
and consider the collineations that fix the $\boxempty$. Since in
addition to the previous case we also have to impose 
\begin{equation}
\left(e,e\right)\mapsto\left(A\left(e\right),B\left(e\right)\right)=\left(e,e\right),
\end{equation}
and since $e*e=e$, then $C\left(e\right)=e$ and $A=B=C$. Therefore
$A$ is an automorphism of $\mathbb{\mathcal{O}}$. We then have the
following 
\begin{prop}
Collineations that fix every point of $\boxempty$ are transformations
of the type 
\begin{align}
\left(x,y\right) & \mapsto\left(A\left(x\right),A\left(y\right)\right)\\
\left(s\right) & \mapsto\left(A\left(s\right)\right)\\
\left(\infty\right) & \mapsto\left(\infty\right)
\end{align}
where $A$ is an automorphism of $\mathbb{\mathcal{O}}$. 
\end{prop}

\begin{cor}
The group of collineations $\Gamma\left(\boxempty,\mathcal{O}\right)$
that fix $\left(0,0\right)$, $\left(e,e\right)$, $\left(0\right)$
and $\left(\infty\right)$ is isomorphic to $\text{Aut}\left(\mathcal{O}\right)\cong SL_{3}\left(\mathbb{R}\right)$. 
\end{cor}

We are now interested in studing the Lie algebras of the previous
groups. 
\begin{prop}
The Lie algebra $Lie\left(\Gamma\left(\triangle,\mathcal{O}\right)\right)$
of the group of collineation that fixes $\left(0,0\right),\left(0\right)$\textup{\emph{and}}\textup{
$\left(\infty\right)$ }\textup{\emph{is}}\emph{ } 
\begin{equation}
\mathfrak{tri}\left(\mathcal{O}\right)=\left\{ \left(T_{1},T_{2},T_{3}\right)\in\mathfrak{so}\left(\mathcal{O}\right)^{3}:T_{1}\left(x*y\right)=T_{2}\left(x\right)*y+x*T_{3}\left(y\right)\right\} ,
\end{equation}
while the Lie algebra $Lie\left(\Gamma\left(\boxempty,\mathcal{O}\right)\right)$
of the group of collineation that fixes $\left(0,0\right)$, $\left(e,e\right)$,
$\left(0\right)$ and $\left(\infty\right)$is

\begin{equation}
\mathfrak{der}\left(\mathcal{O}\right)=\left\{ A\in\mathfrak{so}\left(\mathcal{O}\right):A\left(x*y\right)=A\left(x\right)*y+x*A\left(y\right)\right\} .
\end{equation}
\end{prop}

\begin{proof}
$\Gamma\left(\triangle,\mathcal{O}\right)$ is a Lie group since it
is a closed subgroup of the Lie group of collineations. We will find
directly its Lie algebra considering the elements $A,B,C\in\Gamma\left(\triangle,\mathcal{O}\right)$
in a neighbourhood of the identity and writing them as 
\[
\left(A,B,C\right)\longrightarrow\left(\text{Id}+\epsilon T_{1},\text{Id}+\epsilon T_{2},\text{Id}+\epsilon T_{3}\right)
\]
where $T_{1},T_{2},T_{3}\in\mathfrak{so}\left(\mathcal{O}\right)$.
Imposing the condition $A\left(x*y\right)=B\left(x\right)*C\left(y\right)$
and then we obtain 
\begin{align}
\left(\text{Id}+\epsilon T_{1}\right)\left(x*y\right) & =\left(\text{Id}+\epsilon T_{2}\right)\left(x\right)*\left(\text{Id}+\epsilon T_{3}\right)\left(x\right),
\end{align}
which, considering $\epsilon^{2}=0$, yields to 
\begin{equation}
T_{1}\left(x*y\right)=T_{2}\left(x\right)*y+x*T_{3}\left(y\right).
\end{equation}
As for the second part of the theorem, it suffices to impose $T_{1}=T_{2}=T_{3}$. 
\end{proof}
Defining the Spin group on Okubo algebras $\text{Spin}\left(\mathcal{O}\right)$,
following \cite{ElDuque Comp}, as 
\begin{equation}
\text{Spin}\left(\mathcal{O}\right)=\left\{ \left(A,B,C\right)\in O^{+}\left(\mathcal{O}\right)^{3}:A\left(x*y\right)=B\left(x\right)*C\left(y\right)\,\,\,\forall x,y\in\mathcal{O}\right\} ,
\end{equation}
where $O^{+}\left(\mathcal{O}\right)$ is the connected component
with the identity, we then recover the identification 
\begin{align}
\Gamma\left(\triangle,\mathcal{O}\right)\cong & \text{Spin}\left(\mathcal{O}\right),\\
\Gamma\left(\boxempty,\mathcal{O}\right)\cong & \text{Aut}\left(\mathcal{O}\right)\cong SU\left(3\right),
\end{align}
and, passing to Lie algebras, we obtain 
\begin{align}
Lie\left(\Gamma\left(\triangle,\mathcal{O}\right)\right)\cong & \mathfrak{tri}\left(\mathcal{O}\right)\cong\mathfrak{su}\left(3\right),\\
Lie\left(\Gamma\left(\boxempty,\mathcal{O}\right)\right)\cong & \mathfrak{der}\left(\mathcal{O}\right)\cong\mathfrak{su} \left(3\right).
\end{align}
We then have a perfect analogy with the Octonionic case where the
subgroups $\Gamma\left(\triangle,\mathbb{O}\right)$ and $\Gamma\left(\boxempty,\mathbb{O}\right)$
are given by

\begin{align}
\Gamma\left(\triangle,\mathbb{O}\right)\cong & \text{Spin}\left(\mathbb{O}\right)\cong\text{Spin}_{8}\left(\mathbb{R}\right),\\
\Gamma\left(\boxempty,\mathbb{O}\right)\cong & \text{Aut}\left(\mathbb{O}\right)\cong\text{G}_{2\left(-14\right)},
\end{align}
and the respective Lie algebras are identified with 
\begin{align}
Lie\left(\Gamma\left(\triangle,\mathbb{O}\right)\right)\cong & \mathfrak{tri}\left(\mathbb{O}\right)\cong\mathfrak{so}\left(\mathbb{\mathbb{O}}\right),\\
Lie\left(\Gamma\left(\boxempty,\mathbb{O}\right)\right)\cong & \mathfrak{der}\left(\mathbb{O}\right)\cong\mathfrak{g}_{2}.
\end{align}

\section{Conclusions and future developments}

In this work we defined for the first time an affine and projective
plane over the Okubo algebra. It is mesmerizing and surprising how
the analogy between the alternative algebra of Octonions and the flexible
Okubo algebra works perfectly with only few changes. Even though,
these new affine and projective planes share a numerous proprieties
with the classic planes, it is also worth noting that their geometry
is far to be completey settled. Jordan algebras over three by three
matrices have a perfect analogy with the projective plane over Octonions.
It would be then definitely interesting to analyze is an equivalent
structure can be realized in the case of Okubo's algebras. Moreover,
interesting geometries arise from the complex Okubo algebra which
is not a division algebra but whose divisors of zero enjoy interesting
projective proprieties that we plan to cover in our next work.

\section{Acknowledgment}

\bigskip{}

$\dagger$\noun{Departamento de Matemática, }\\
\noun{Universidade do Algarve, }\\
\noun{Campus de Gambelas, }\\
\noun{8005-139 Faro, Portugal} 
\begin{verbatim}
a55499@ualg.pt.
\end{verbatim}
$\ddagger$\noun{Dipartimento di Scienze Matematiche, Informatiche
e Fisiche, }\\
\noun{Università di Udine, }\\
\noun{Udine, 33100, Italy} 
\begin{verbatim}
francesco.zucconi@uniud.it
\end{verbatim}

\end{document}